\documentclass{article}

\usepackage{amsmath}
\usepackage{url}
\RequirePackage{graphicx}
\RequirePackage{pifont,latexsym,ifthen,amsthm,rotating,calc,booktabs}
\RequirePackage{amsfonts,amssymb,amsbsy,amsmath}

\theoremstyle{definition}
\newtheorem{theorem}{Theorem}

\newtheorem{corollary}[theorem]{Corollary}
\newtheorem{proposition}[theorem]{Proposition}
\newtheorem{definition}[theorem]{Definition}

\newtheorem{problem}[theorem]{Problem}

\newcommand\point{\mathrm{pt}}

\makeatletter
\@twosidetrue
\makeatother

\begin{document}

\title{A complete $g$-vector for convex polytopes}

\author{Jonathan Fine\\
15 Hanmer Road, Milton Keynes, MK6 3AY, United Kingdom\\
email: jfine@pytex.org}

\date{10 January 2010}
\maketitle

\begin{abstract}

\noindent
We define an extension of the toric (middle perversity intersection
homology) $g$-vector of a convex polytope $X$. The extended $g(X)$
encodes the whole of the flag vector $f(X)$ of $X$, and so is called
complete.  We find that for many examples that $g_k(X)\geq 0$ for most
$k$ (independent of $X$).
\end{abstract}

\section{Definitions}

In this section we will show that the equation
\begin{equation*}
g_k(\>(C,D)^\ell(\mathrm{pt})\>)\>\> = \>\> (k \leq \ell)
\end{equation*}
defines an extension of toric $g$-vector.

First, we define some terms.  Let $P$ be a convex polytope, or
polytope for short.  Throughout we will use $CP$ to denote the
\emph{cone} (or prism) on $P$, and $IP$ to denote the \emph{cylinder}
(or prism) on $P$.  Our definition relies on

\begin{theorem}[Generalised Dehn-Sommerville, \cite{baybil85:gen}]
The flag vectors of the polytopes $W(\point)$, where $W$ is a word in
$C$ and $IC$, provide a basis for the span of all convex polytope flag
vectors.
\end{theorem}

Thus, any linear function $\alpha$ on polytope flag vectors is
determined by its values $\alpha(f(W(\point)))$, or
$\alpha(W(\point))$ for short, where $W$ ranges over all words in $C$
and $IC$.  Conversely, the values $\alpha(W(\point))$ determine
$\alpha$.  We find it helpful to introduce $D = IC - CC$.  This is an
operator on formal sums of convex polytopes.  If $\alpha$ is a
function defined on convex polytopes and $P = \sum \lambda_iP_i$ is a
formal sum of convex polytopes then $\alpha(P) = \sum
\lambda_i\alpha(P_i)$ extends $\alpha$ to formal sums.

We will use the following form of generalised Dehn-Sommerville.

\begin{corollary}
Suppose $\alpha(W(\point))$ is known for all words $W$ in $C$ and $D$.
Then $\alpha$ has a unique extension, as a linear function of flag
vectors, to all convex polytopes.
\end{corollary}

We are particularly interested in occurrences, in $W$, of $CD$.  (This
is because they creating new singularities.  The `polytopes' $X$ and
$DX$ have very similar minimal stratifications, as do $CX$ and $CCX$.
But $CDX$ is sure to a new stratum, namely the apex.)

\begin{proposition}
Suppose $W$ is a word in $C$ and $D$.  Then $W$ has a unique
expression in the form
\begin{equation}
\label{eqn:CDk}
  D^{i_0}C^{j_0}\> CD \> D^{i_1}C^{j_1} \> CD \> \ldots \> D^{i_r}C^{j_r}
\end{equation}
where the $i_n, j_n$ lie in $\mathbb{N}$.
\end{proposition}

\begin{proof}
Split $W$ at the occurences of $CD$ as a subword.  The resulting
pieces do not contain $CD$ and so must be of the form $D^iC^j$.
\end{proof}

\begin{definition}
If $k =((i_0,j_0),\ldots,(i_r,j_r))$ is a sequence of $r$ pairs of
natural numbers then we will say that $k$ is a \emph{Fibonacci index}
of \emph{order} $r+1$ and of \emph{degree} $2\sum i_n + \sum j_n + 3r$.
\end{definition}

\begin{definition}
For $k$ a Fibonacci index we let $(C, D)^k$ denote the word~(\ref{eqn:CDk}).
\end{definition}

Often, we will use $[{}^{i_0}_{j_0}{}^{i_1}_{j_1}]$ or
$[i_0,i_1;j_0,j_1]$ to denote a Fibonacci index.  Concatenation,
denoted by juxtaposition, can be used to combine indexes.  Thus
$[a;b][i_0,i_1;j_0,j_1]=[a,i_0,i_1;b,j_0,j_1]$.

\begin{definition}
The \emph{singularity partial order} $u<v$ on Fibonacci indexes is
generated by the \emph{shifting}
\begin{equation*}
\begin{bmatrix}a\\c\end{bmatrix}
\begin{bmatrix}b\\d+1\end{bmatrix}
<
\begin{bmatrix}a\\c+1\end{bmatrix}
\begin{bmatrix}b\\d\end{bmatrix}
\end{equation*}
and \emph{splitting}
\begin{equation*}
\begin{bmatrix}a+b+1\\c+1\end{bmatrix}
<
\begin{bmatrix}a\\0\end{bmatrix}
\begin{bmatrix}b\\c\end{bmatrix}
\end{equation*}
relations, together with the \emph{concatenation}
\begin{equation*}
u < v \implies xuy < xvy
\end{equation*}
relation, where $u, v, x, y$ are any Fibonacci indexes.
\end{definition}

\begin{definition}
For $k$ a Fibonacci index $g_k(X)$ is the linear function on polytope
flag vectors determined by
\begin{equation}
\label{eqn:g_k}
g_k(\>(C,D)^\ell(\mathrm{pt})\>)\>\> = \>\> (k \leq \ell)
\end{equation}
where $(k\leq\ell)$ is $1$ if $k\leq\ell$ and $0$ otherwise.
\end{definition}

\begin{proposition}
If $d = \dim X$ and $k = [i;2d-2i]$ then $g_i(X) = g_k(X)$.
\end{proposition}

\begin{proof}
It is easily seen that $k\leq \ell$ iff $(C,D)^\ell$ contains exactly
$i$ occurences of $D$.  We have $g_i(\point) = \delta_{0,i}$
(Kronecker delta).  It follows from basic properties of
$g_i$~\cite{sta87:gen} that $g_i(CP)$ = $g_i(P)$ and $g_{i+1}(DP) =
g_i(D)$.  The result now follows.
\end{proof}

\begin{theorem}
The flag vector $f$ of convex polytopes is a linear function of the
extended $g$-vector.
\end{theorem}

\begin{proof}
The words in $C, D$ provide a basis for the flag vectors, and $g$ is
computed from this basis by applying a lower-triangular matrix.
\end{proof}

\section{Calculations}

We denote by $BP$ the bipyramid on $P$.  In the propositions that
follow, $k$ is a degree $d$ word Fibonacci index, and $X = W(\point)$
for $W$ a length $d$ word in $C$, $I$ and $B$.  The
calculations~\cite{fine:python-hvector} been done using
polymake~\cite{polymake}.

\begin{proposition}
Assume $d \leq 4$. Then $g_k(X)\geq 0$.
\end{proposition}

\begin{proposition}
Let $d = 5$. If $g_k(X)<0$ then $k =[00;11]$, and for this $k$ we
have $g_k(BIC^3(\point))<0$.
\end{proposition}

\begin{proposition}
Let $d = 6$.  If $g_k(X)< 0$ then $k = [00;21]$, and for this $k$
we have $g_k(B^4C^2(\point))<0$.
\end{proposition}

\begin{proposition}
Let $d = 7$.  If $g_k(X)< 0$ then $k \in \{[00;13], [00;31],
[10;11]\}$, and for these $k$ we have $g_k(B^2ICIC^2(\point))<0$.
\end{proposition}

\begin{proposition}
Let $d = 8$.  If $g_k(X)< 0$ then $k \in \{ [00;23], [00;41],[10;21],
[000;011]\}$, and for these $k$ we have $g_(B^4CIC^2(\point)) < 0$.
\end{proposition}

The exceptional $k$ seem to be following a pattern, but it is not
clear what that pattern is.  The author intends to perform further
calculations using, for example, joins of polytopes.

The numbers $g_k(X)$ are small compared to the components of $f(X)$.
For $X=BCBCBCBC(\mathrm{pt})$ we have
\begin{align*}
f(X) &=(1, 13, 73, 232, 697, 458, 1840, 2764, 578, 2917, 5863, 
5866, 17624, 459, 2802, \\
&\qquad\qquad 7087, 9503, 28552, 7117, 28597, 42960, 213, 1538, 4729, 8021, 24100,\\ 
&\qquad\qquad 8089, 32505, 48832, 4833, 24399, 49050, 49077, 147456)\\
g(X) &= (1, 197, -158, 167, 176, -14, 48, 34, 48, 30, 2, 88, 35, 
36, 33, 11, 18, 4, 77, 3,\\
&\qquad\qquad 25, 37, 33, 10, 17, 4, 16, 0, 3, 78, 3, 2, 26, 1)
\end{align*}

\section{Problems}

The major problem arising from the definition and calculations is, of
course, to find the exceptional $k$ and to prove that for the others
$g_k(X)\geq 0$.  This is known for $k=[i;d-2i]$ by virtue of middle
perversity intersection homology and the decomposition theorem.  For
this reason, we expect the proof of $g_k(X)\geq 0$ for
non-exceptional $k$ to be hard and rewarding.

Here is a more accessible and purely combinatorial problem.  But first
we need a definition.  Recall that each convex polytope polytope $X$
has a polar $X^\vee$, and $f(X^\vee)$ is a linear function of $f(X)$.
Thus, $g^\vee(X) = g(X^\vee)$ is a linear function of $g(X)$.

\begin{definition}
Suppose $X$ is either a convex $d$-polytope or a formal sum of such.
Suppose $S$ is a subset of the $d$-indexes.  If $g_k(X)\geq 0$ and
$g^\vee_k(X)\geq 0$ for all $k\in S$ then we will say $S$ is
\emph{effective} on $X$.
\end{definition}

\begin{problem}
Here $X$ denotes a formal sum of convex polytopes (or a point in the
span of convex polytope flag vectors).  The problem is to find index
sets $\{S_n\}$ such that if $S_d$ is effective on $X$ then it follows
that $S_{d+1}$ is effective on $CX$, $IX$ and $BX$.
\end{problem}

A suitable solution to this problem would remove the limit $d\leq 8$
on the calculations in the previous section, and help us understand
the exceptional indexes.

\bibliographystyle{amsplain} 
\bibliography{ehv-note2.bib}

\end{document}